\documentclass[10pt]{amsart}
\usepackage{amsmath}
\usepackage{amsfonts}
\usepackage{amsthm}
\usepackage{amssymb}

\newtheorem{theorem}{Theorem}
\newtheorem*{theorem*}{Theorem}
\numberwithin{equation}{section}
\numberwithin{theorem}{section}

\newtheorem*{acknowledgement*}{Acknowledgement}
\newtheorem*{claim*}{Claim}
\newtheorem*{definition*}{Definition}

\newtheorem{corollary}[theorem]{Corollary}

\newtheorem{lemma}[theorem]{Lemma}

\newtheorem{proposition}[theorem]{Proposition}

\newcommand{\RR}[0]{\mathbb{R}}

\newcommand{\CC}[0]{\mathbb{C}}
\newcommand{\pd}[2]{\frac{\partial #1}{\partial#2}}

\newcommand{\delb}[0]{\overline{\nabla}}
\newcommand{\delt}[0]{\widetilde{\nabla}}

\newcommand{\Rc}[0]{\operatorname{Rc}}
\newcommand{\Rm}[0]{\operatorname{Rm}}

\newcommand{\pdtau}[0]{\pd{}{\tau}}

\newcommand{\lam}{\lambda}

\newcommand{\Isom}[0]{\operatorname{Isom}}
\newcommand{\Cc}[0]{\mathcal{C}}
\newcommand{\Ac}[0]{\mathcal{A}}

\title[Isometries of asymptotically conical shrinkers]{Isometries of asymptotically conical \\shrinking Ricci solitons}

\author{Brett Kotschwar}
\address{Arizona State University, Tempe, AZ, USA}
\email{kotschwar@asu.edu}

\author{Lu Wang}
\address{University of Wisconsin, Madison, WI, USA}
\email{luwang@math.wisc.edu}

\thanks{The first author was supported in part by Simons Foundation grant \#359335. The second author was partially supported by the NSF Grants DMS-1811144 and DMS-1834824, an Alfred P. Sloan Research Fellowship, the funding from the Wisconsin Alumni Research Foundation, and a Vilas Early Career Investigator Award. The second author was also supported by a von Neumann Fellowship by the Institute for Advanced Study: This material is based upon work supported by the Z\"{u}rich Insurance Company Membership and the NSF Grant DMS-1638352.}

\begin{document}
\begin{abstract} 
We show that if a shrinking soliton is asymptotic to a cone along an end then the isometry group of the cross-section of the cone
embeds in the isometry group of the end of the shrinker. We also provide sufficient conditions for the isometries of the end to extend to
the entire shrinker.

\end{abstract}
\maketitle

\section{Introduction}
Shrinking Ricci solitons model the geometry of solutions to the Ricci flow in the vicinity of a developing singularity.
All known
complete noncompact shrinking Ricci solitons which are not locally reducible as a product are smoothly asymptotic to a cone at infinity \cite{DancerWang, FeldmanIlmanenKnopf, Yang}. In four dimensions, there is growing evidence to suggest that the asymptotically conical shrinkers are the only nontrivial
complete noncompact examples. Their classification is vital
to the understanding of the long-term behavior of the equation and to potential future topological applications. 

In our previous work
\cite{KotschwarWangConical}, we have shown that these cones essentially determine the shrinker: whenever two shrinkers are asymptotic to the same cone
along some ends of each, they must be isometric near infinity on those ends. This reduces the classification to that of their possible asymptotic cones. A reasonable first step toward an understanding of what cones can occur is to identify the geometric features which an asymptotically conical shrinker and its asymptotic cone must share in common.

The purpose of this note is to detail the application of the uniqueness theorem in \cite{KotschwarWangConical} to the relationship between the isometry group of cone and shrinker.  A direct application of that theorem implies that any symmetry of the asymptotic cone must, at least, be reflected in a symmetry of the 
shrinker on some neighborhood of infinity. The precise statement is this: 
if a shrinker $(M, g, f)$ is asymptotic to a cone $\Cc$
along an end $V\subset (M, g)$, then, for any isometry $\phi$ of $\Cc\setminus\{\mathcal{O}\}$, there is an end $W\subset V$ and a diffeomorphism 
$F: W\to \Cc\setminus \{r \leq r_0\}$
such that $F^{-1}\circ \phi \circ F$ is an isometry of $(W, g|_W)$. The argument in \cite{KotschwarWangConical}
does not provide a completely effective bound on the size of the end $W$, however, and, by itself, does not preclude that $F$ and $W$ may 
depend on the isometry $\phi$.  
 
In this paper, we demonstrate
that a uniform choice of $F$ and $W$ can be made for which the above correspondence provides an isomorphism between
the vertex-fixing isometries of the cone and the isometries of the restriction of the shrinker to $W$.
When the shrinker is complete, and the end is large enough, topologically,
we show that the stabilizer of the vertex of the cone embeds into the isometry group of the entire shrinker. 
This topological condition is satisfied,
 for example, by the end of any complete K\"ahler shrinking soliton.

We will need to introduce some notation to state our main result. Let $\Cc^{\Sigma}$ denote the cone over the smooth compact $(n-1)$-dimensional manifold $(\Sigma, g_\Sigma)$,
with vertex $\mathcal{O}$. For $a \geq 0$, let $\Cc_a^{\Sigma} =  (a, \infty)\times \Sigma$, and denote by $\hat{g} = dr^2+r^2g_{\Sigma}$
the conical metric on $\Cc_0^{\Sigma} = \Cc^{\Sigma}\setminus\{\mathcal{O}\}$. Finally, for $\lambda > 0$, let $\rho_{\lambda}:\Cc_0^{\Sigma}\to \Cc_0^{\Sigma}$ denote the dilation map $\rho_{\lambda}(r, \sigma) = (\lambda r, \sigma)$.

Following \cite{KotschwarWangConical}, we will then say that a Riemannian manifold $(M, g)$ is \emph{asymptotic to $\Cc^{\Sigma}$ along the end $V\subset (M, g)$}
if, for some $a > 0$, there is a diffeomorphism $F:\Cc_a^{\Sigma}\to V$ such that $\lam^{-2}\rho_{\lambda}^*F^\ast g\to \hat{g}$
as $\lambda \to \infty$
in $C^2_{\emph{loc}}(\Cc_0^{\Sigma}, \hat{g})$. We will say that a soliton $(M, g, f)$ is asymptotic to $\Cc^{\Sigma}$ along $V$ when $(M, g)$ is. Here, by \emph{end},
we mean an unbounded connected component of the complement of a compact set.  

By \cite{MunteanuWangConical1}, \cite{MunteanuWangConical2}, any complete shrinker which satisfies $|\Rc|(x)\to 0$ as $x\to \infty$ 
will be asymptotic to a cone on each of its ends in the above sense. 
In dimension four, the same is true only assuming the scalar curvature tends to zero at infinity. The uniqueness of the asymptotic cone is discussed in 
\cite{ChowLuAsymptoticCone}.

The first nontrivial examples of complete asymptotically conical shrinkers were exhibited by Feldman-Ilmanen-Knopf \cite{FeldmanIlmanenKnopf}. 
Their construction uses an ansatz of $\operatorname{U}(m)$-symmetry to produce complete K\"ahler shrinkers on the tautological line bundle of $\CC P^{m-1}$ for each $m\geq 2$. This construction was generalized by Dancer-Wang \cite{DancerWang} and Yang \cite{Yang} to produce further complete K\"ahler asymptotically conical  examples on line bundles over products of K\"ahler-Einstein metrics with positive scalar curvature.

\begin{theorem}\label{thm:isom}
 Suppose the shrinking soliton $(M, g, f)$ is asymptotic to $\mathcal{C}^{\Sigma}$ along the end $V\subset (M, g)$.  Then
 there is $a\geq 0$, an end $W$ of $(M, g)$ with $W\subset V$, and a diffeomorphism $F: W \to \Cc_a^{\Sigma}$ such that 
 $\gamma \mapsto F^{-1}\circ (\operatorname{Id}\times\gamma) \circ F$ is an isomorphism from $\operatorname{Isom}(\Sigma, g_{\Sigma})$ onto $\operatorname{Isom}(W, g|_W)$.
\end{theorem}
Here, the conclusion does not require that $(M, g)$ be either complete or connected at infinity. However, when $(M, g)$ is complete, it is reasonable to ask
whether the isometries on $W$ extend to isometries on $M$. The answer is yes at least when the fundamental group of $V$ surjects onto that of $M$.
This is a straightforward variation on the classical continuation argument
for local isometries on simply-connected real-analytic manifolds; see Theorem \ref{thm:globalext}.

\begin{theorem}\label{thm:isomg}
 Suppose that $(M, g, f)$ is complete and asymptotic to the cone $\Cc^{\Sigma}$ along the end $V\subset (M, g)$. If the homomorphism 
 $\pi_1(V, x_0) \to \pi_1(M, x_0)$ induced by inclusion
 is surjective for some $x_0\in V$, then $\Isom(\Sigma, g_{\Sigma})$ embeds into $\Isom(M, g)$.
\end{theorem}
The technical hypothesis on $V$ is always met when $(M, g)$ is K\"ahler: indeed, a complete noncompact K\"ahler shrinker is connected at infinity \cite{MunteanuWangKaehler},
so this is true of the lift of $(M, g, f)$ to the universal cover $\tilde{M}$ of $M$ (which is also a complete K\"ahler shrinker). Since $\pi_1(M)$ is necessarily finite \cite{Wylie}, the preimage of $V$ in the universal cover $\tilde{M}$ must be connected. 
See Section 4 of \cite{KotschwarKaehler} for details. In fact, by the main result of \cite{KotschwarKaehler},
it is actually only necessary to assume that the cone is K\"ahler.
\begin{corollary} Suppose $(M, g, f)$ is complete and asymptotic to $\Cc^{\Sigma}$ along some end. If $(\Cc_0^{\Sigma}, \hat{g})$ is K\"ahler, so is $(M, g)$ and 
$\Isom(\Sigma, g_{\Sigma})$ embeds in $\Isom(M, g)$. Moreover, $(M, g, f)$ is the unique complete shrinker asymptotic to $\Cc^{\Sigma}$.
\end{corollary}
We only use the K\"ahler property here to conclude that the universal cover of the shrinker is connected at infinity. It is an interesting question
whether every complete simply-connected shrinker with more than one end must split as a product.

\begin{acknowledgement*}  We thank Ronan Conlon for his interest and for useful discussions regarding the application of \cite{KotschwarWangConical}
to the isometries of the end.
\end{acknowledgement*}

\section{The self-similar solution on an asymptotically conical end}
We will assume below that our shrinking solitons $(M, g, f)$ are normalized to satisfy 
\begin{equation}\label{eq:grs}
 \Rc(g) + \nabla\nabla f = \frac{g}{2}, \quad R + |\nabla f|^2 = f.
\end{equation}
When $(M, g, f)$ is asymptotic to a cone $\Cc^{\Sigma}$ along an end $V\subset (M, g)$, it is shown in Section 2 of \cite{KotschwarWangConical} that
there is an end $W\subset V$ and a solution $g(t)$ to the Ricci flow on $W$ which interpolates between $g$ at $t= -1$ and (an isometric copy of)
$\hat{g}$ at $t=0$. This solution is smooth on $W\times[-1, 0]$, though only self-similar for $-1 \leq t < 0$.
It will be convenient in what follows to work in terms of the parameter $\tau = -t$ and regard $g= g(\tau)$ as a solution to the \emph{backward Ricci flow}
\begin{equation}\label{eq:brf}
 \pd{}{\tau} g = 2 \Rc(g).
\end{equation}
This solution transforms Theorem \ref{thm:isom} into a question of the preservation of isometries along the flow \eqref{eq:brf}.

\begin{proposition}[Proposition 2.1, \cite{KotschwarWangConical}]
\label{prop:brf}
Suppose the shrinker $(M, \tilde{g}, \tilde{f})$ is asymptotic to $\Cc^{\Sigma}$ along the end $V\subset (M, \tilde{g})$.
Then there exists $r_0 > 0$ and a diffeomorphism $F: \Cc_{r_0}\to W$ onto an end $W\subset V$
such that $\bar{g} = F^*\tilde{g}$, $\bar{f} = F^*\tilde{f}$ satisfy the following properties.
\begin{enumerate}

\item[(1)] The solution $\Phi$ to 
 \begin{equation}
 \label{eq:phisys}
  \pd{\Phi}{\tau} = -\frac{1}{\tau}\delb \bar{f} \circ \Phi, \quad \Phi_1 = \operatorname{Id},
 \end{equation}
is well-defined on $\Cc_{r_0}^{\Sigma}\times (0, 1]$,
and the maps $\Phi_{\tau}:\Cc_{r_0}^{\Sigma}\to \Cc_{r_0}^{\Sigma}$ are each injective local diffeomorphisms for $\tau \in (0, 1]$.

\item[(2)] The family of metrics $g(\tau) = \tau \Phi_{\tau}^*\bar{g}$ is a smooth solution to \eqref{eq:brf} on $\Cc_{r_0}^{\Sigma}\times (0, 1]$
and converges smoothly to $\hat{g}$ on $\overline{\Cc_{a}^{\Sigma}}$ for all $a > r_0$ as $\tau \to 0$. Moreover, there is a constant $K_0$ such that
\begin{align}
\label{eq:curvdecay}
\sup_{\Cc_{r_0}^{\Sigma}\times [0,1]} \left(r^{m+2}+1\right)|\nabla^{(m)}\Rm(g(\tau))| & \leq K_0.
\end{align}
Here $|\cdot| = |\cdot|_{g(\tau)}$ and $\nabla = \nabla_{g(\tau)}$ denote the norm and the Levi-Civita connection associated to the metric $g = g(\tau)$.

\item[(3)] If $f$ is the function on $\Cc_{r_0}^{\Sigma}\times (0,1]$ defined by 
$f(\tau)=\Phi_\tau^\ast \bar{f}$, then $\tau f$ converges smoothly as $\tau \to 0$
to $r^2/4$ on  $\overline{\Cc^{\Sigma}_{a}}$ for all $a > r_0$, and satisfies
\begin{align}
\label{eq:fid0}
r^2-\frac{N_0}{r^{2}} \le 4 \tau f(r, \sigma, \tau) \le r^2 + \frac{N_0}{r^{2}}, \quad \tau \nabla f = \frac{r}{2}\pd{}{r},
\end{align}
on $\Cc^{\Sigma}_{r_0}\times (0, 1]$ for some constant $N_0 > 0$.
\item[(4)] Together,  $g= g(\tau)$ and $f = f(\tau)$ satisfy 
\begin{equation}
 \label{eq:grstau}
 \Rc(g) + \nabla\nabla f = \frac{g}{2\tau}, \quad R + |\nabla f|^2 = \frac{f}{\tau}
\end{equation}
on $\Cc_{r_0}^{\Sigma}\times (0, 1]$.
\end{enumerate}
Here $r$ denotes the radial distance $r(x) = d(\mathcal{O}, x)$ on $\Cc^{\Sigma}$.
\end{proposition}
\begin{proof}
The second identity in \eqref{eq:fid0} follows from the smooth convergence
of $g$ to $\hat{g}$ and of $\tau f$ to $r^2/4$ as $\tau\to 0$ and that $\tau \nabla f = \operatorname{grad}_{g(\tau)}f(\tau)$ is independent of $\tau$. Indeed,
\begin{equation*}
 \pd{}{\tau}\left( \tau \operatorname{grad}_{g(\tau)} f(\tau)\right) = \pdtau \Phi_{\tau}^{*}\left(\operatorname{grad}_{\bar{g}} \bar{f}\right) =
 -\frac{1}{\tau} \left.\pd{}{s}\right|_{s= -\ln\tau} \varphi_s^{*}\left(\operatorname{grad}_{\bar{g}} \bar{f}\right) = 0
\end{equation*}
on $\Cc_{r_0}^{\Sigma}\times (0, 1]$, where $\varphi_s:\Cc^{\Sigma}_{r_0}\to \Cc^{\Sigma}_{r_0}$ is the family of local diffeomorphisms generated by $\operatorname{grad}_{\bar{g}} \bar{f}$.  Here, we use $\Phi_{\tau}^*$ and $\varphi_s^*$ to denote $(\Phi_{\tau}^{-1})_{*}$ and $(\varphi_s^{-1})_{*}$, respectively.
The other assertions are part of Proposition 2.1 in \cite{KotschwarWangConical}.
\end{proof}

Proposition \ref{prop:brf} allows us to work with a soliton structure $(\Cc^{\Sigma}_{r_0}, \bar{g}, \bar{f})$ which flows directly to the cone under the 
Ricci flow. We will say that a shrinker $(\Cc^{\Sigma}_{r_0}, \bar{g}, \bar{f})$ satisfying properties (1)-(4) of Proposition \ref{prop:brf} is \emph{dynamically asymptotic} to $(\Cc^{\Sigma}_{r_0}, \hat{g})$.
The correspondence expressed in Theorem \ref{thm:isom} can be improved for shrinkers that have been normalized in this sense.

\begin{theorem}\label{thm:isomp} Suppose the shrinking soliton $(\Cc^{\Sigma}_{r_0}, \bar{g}, \bar{f})$ is dynamically
asymptotic to $(\Cc^{\Sigma}_{r_0}, \hat{g})$. Then $\Isom(\Cc^{\Sigma}_{r_0}, \hat{g}) = \Isom(\Cc^{\Sigma}_{r_0}, \bar{g})$.
 \end{theorem}
Here, the identity asserted between $\Isom(\Cc^{\Sigma}_{r_0}, \hat{g})$ and  $\Isom(\Cc^{\Sigma}_{r_0}, \bar{g})$
is an equality of sets (that is, as subsets of $\operatorname{Diff}(\Cc_{r_0}^{\Sigma})$).

\begin{proof}[Proof of Theorem \ref{thm:isom}, assuming Theorem \ref{thm:isomp}]
Suppose $(M, g, f)$ is asymptotic to $\Cc^{\Sigma}$
along the end $V\subset (M, g)$. By Proposition \ref{prop:brf}, there is $r_0 > 0$, an end $W\subset V$, and a diffeomorphism $F: \Cc^{\Sigma}_{r_0}\to W$ such that, with
$\bar{g} = F^*g$, $\bar{f} = F^*f$, the soliton structure $(\Cc^{\Sigma}_{r_0}, \bar{g}, \bar{f})$ is dynamically asymptotic to $\Cc^{\Sigma}$. 

By Theorem 2.2, then, $\phi\mapsto F \circ \phi\circ F^{-1}$ is an isomorphism
of $\Isom(\Cc^{\Sigma}_{r_0}, \hat{g})$ onto $\Isom(W, g|_W)$. However, an isometry $\phi: (\Cc^{\Sigma}_{r_0}, \hat{g})\to (\Cc^{\Sigma}_{r_0}, \hat{g})$ is
the restriction of a vertex-preserving isometry of $\Cc^{\Sigma}$ to $\Cc^{\Sigma}_{r_0}$, and so $\phi = \operatorname{Id}\times \gamma$ for some $\gamma\in \Isom(\Sigma, g_{\Sigma})$. 
\end{proof}

In particular, any Killing vector field on the cross-section  $(\Sigma, g_{\Sigma})$ corresponds to a Killing vector field of the soliton metric $g$ on $\Cc^{\Sigma}_{r_0}$.
In this paper, we are only interested in the isometries of the cone which are induced by those on the cross-section. These isometries are precisely those which
fix the vertex of the cone. If the cone $\Cc^{\Sigma}$ admits a isometry which does
not fix the vertex, then $\Cc^{\Sigma}$ is smooth and in fact isometric to Euclidean space. By \cite{KotschwarWangConical}, the only shrinker which is asymptotic to such a cone is itself (an end of) the Gaussian shrinker.

\section{The analytic structure associated to a solution to the Ricci flow}
We next wish to show that the end of an asymptotically conical shrinker and the cone to which it is dynamically asymptotic are real-analytic relative to a common real-analytic structure.
This is particularly convenient to address from within the parabolic framework established in the previous section. This framework realizes both the end of the soliton $(\Cc^{\Sigma}_{r_0}, \bar{g})$ and the end of the cone $(\Cc^{\Sigma}_{r_0}, \hat{g})$
as  
time-slices of the same smooth Ricci flow, and reduces the problem to a matter of sharpening the usual statement of instantaneous real-analyticity
for that equation.

Most of this section concerns general smooth solutions to the Ricci flow and is independent of the rest of the paper. We will temporarily use $g(t)$ to denote an
arbitrary smooth solution to the (forward) Ricci flow on 
$M\times [0, T]$. It is a classical theorem of Bando \cite{Bando} that, when $M$ is compact, for each $0 < t \leq T$, $(M, g(t))$ is real-analytic relative to the atlas of normal coordinates associated to $g(t)$. His proof, based on Berstein-type estimates on the covariant derivatives of the curvature tensor, carries over essentially verbatim to the case that $(M, g(t))$ is complete and $\sup_{M\times[0, T]} |\Rm|(x, t) < \infty$.  The estimates in \cite{Bando} were later localized by the first author in \cite{KotschwarAnalyticity} to show that 
for any
 smooth solution to the Ricci flow, $(M, g(t))$ is real-analytic for each $t > 0$. 
 
 Here we point out that these same estimates imply that the metrics $g(t)$ are in fact real-analytic relative to a fixed atlas, viz., the real-analytic structure (i.e., maximal real-analytic atlas) induced by the atlas of normal coordinates taken relative to any one of the $g(t)$. 
 Note that the solutions need not be complete, an aspect important for our application below. For complete solutions of bounded curvature,
 a stronger space-time analyticity result is proven in \cite{KotschwarTimeAnalyticity}. See also \cite{Shao}.  
 
 \begin{theorem}\label{thm:analyticstructure} Suppose $g(t)$ is a smooth solution to the Ricci flow on $M\times [0, T]$. Then there exists a unique real-analytic structure $\Ac$
 relative to which $(M, g(t))$ is analytic for all $t\in (0, T]$. This structure is generated by the atlas of $g(t)$-normal coordinate charts for any $t > 0$. 
 \end{theorem}
 
 In the next section, we will use the following application of the above theorem to the Ricci flow associated to a shrinker on an asymptotically conical end.
 \begin{corollary}\label{cor:analyticcone}
  Suppose $(\Cc^{\Sigma}_{r_0}, g, f)$ is dynamically asymptotic to $(\Cc^{\Sigma}_0, \hat{g})$. Then $(\Cc^{\Sigma}_{r_0}, g)$ and $(\Cc^{\Sigma}_{r_0}, \hat{g})$ are real-analytic relative to a common real-analytic structure. In particular, the cross-section $(\Sigma, g_{\Sigma})$ of the asymptotic cone of a shrinker is real-analytic.
 \end{corollary}
\begin{proof}[Proof of Corollary \ref{cor:analyticcone}]
    We apply the theorem to $\Cc^{\Sigma}_s$ for $s > r_0$. By the compactness of $\partial \Cc_s^{\Sigma}$, the restriction to $\Cc^{\Sigma}_s$ of the solution to \eqref{eq:brf} associated to $(\Cc^{\Sigma}_{r_0}, g, f)$
    is defined for $\tau \in [0, 1+\delta]$, that is, for $t\in [-\delta, 1]$ for some $\delta > 0$. This shows that $g$ and $\hat{g}$ are real-analytic relative 
    to a common real-analytic structure on $\Cc^{\Sigma}_s$ for all $s > r_0$, and hence on $\Cc^{\Sigma}_{r_0}$. Since the conical radial distance function $r$ satisfies the elliptic equation $\hat{\Delta} r^2 = 2n$, it is real-analytic on $\Cc^{\Sigma}_{r_0}$ relative to this structure as well. Using the implicit function
    theorem in the real-analytic category, it follows that $g_{\Sigma}$ is real-analytic.
\end{proof}
 Corollary \ref{cor:analyticcone} shows that there are many smooth compact manifolds $(\Sigma, g_{\Sigma})$ which cannot arise as the cross-section of the asymptotic cone of a shrinker. By contrast,
 there are, as yet, no known restrictions on the asymptotic cones of asymptotically conical expanding solitons. (See, e.g., \cite{DeruelleExpConical}, \cite{LottWilson}, \cite{LottZhang},
 \cite{SchulzeSimon}.) The real-analyticity of solitons was first proven in \cite{Ivey}.

\begin{proof}[Proof of Theorem \ref{thm:analyticstructure}]
 For the uniqueness of the structure $\Ac$, note first that, If $(U, \varphi)$ and $(V, \psi)$ are any overlapping smooth charts in  which the expression of some metric
 $g$ is analytic,
 then the transition map 
 \[
\psi\circ\varphi^{-1}|_{\varphi(U\cap V)}:(\varphi(U\cap V), (\varphi^{-1})^*g)\to (\psi(U\cap V), (\psi^{-1})^*g)
\]
is an isometry between real-analytic metrics
 on open subsets of $\RR^n$ and therefore real-analytic itself. Thus the two charts must belong to the same real-analytic structure. At the same time, according to \cite{KotschwarAnalyticity}, the manifolds $(M, g(t))$ for $t > 0$ are real-analytic relative to the unique real-analytic
 structure $\Ac_t$ which contains the atlas of $g(t)$-normal coordinates. 
 Thus $\Ac$, if it exists, must contain-- and, by maximality, coincide with-- $\Ac_t$ for each $t > 0$.
 
 To prove the existence of $\Ac$, we will show that the structures $\Ac_t$ coincide for $t\in (0, T]$. Since $(0, T]$ is connected, we need only show that $\Ac_t$ is locally constant in $t$. This is a consequence
 of the following claim. 
 \begin{claim*} Let $t_0\in (0, T]$ and suppose the expression $g_{ij}(x, t_0)$ of $g(t_0)$ in the chart $(U, \varphi)$ is real-analytic in $x$.
 For any $x_0 \in U$,
 there exists an $\epsilon > 0$ and a neighborhood $V\subset U$ of $x_0$ such that, for each $t\in (0, T]$ with $|t - t_0| < \epsilon$, the expression of $g(t)$ in the same chart satisfies
 \[
  \sup_{V}|\partial^{(k)}g(\cdot, t)|\leq CL^kk! 
 \]
for all $k\geq 0$ and some $C$, $L$ independent of $k$.
 \end{claim*}
 This claim follows from the local estimate in \cite{KotschwarAnalyticity} and two straightforward combinatorial estimates in  Section 8 of \cite{KotschwarTimeAnalyticity}.
 Fix $x_0 \in U$ and let $V_1 \subset V_0$ be precompact neighborhoods of $x_0$ with $\overline{V_1} \subset V_0 \subset \overline{V_0}\subset U$. 
 By Theorem 1.4 of \cite{KotschwarAnalyticity},
 there are constants $C_0$, $L_0$ depending on $n$, $t_0$, $\max\{T, 1\}$, and $\sup_{\overline{V_0}\times [0, T]} |\Rm|(x, t)$, such that
 \begin{equation}\label{eq:deriv}
    \sup_{(x, t)\in V_1\times [t_0/2, T]} |\nabla^{(k)}\Rm|(x, t) \leq C_0L_0^k (k-2)!
 \end{equation}
for all $k\geq 0$. (The ``lag'' of $-2$ in the factorial above is a trick, due to Lax, to aid the estimation of terms with multiple factors of $\nabla^{(k)}\Rm$.) 

Write $\bar{\nabla} = \nabla_{g(t_0)}$ and $G = \nabla -\bar{\nabla}$
(so, $G_{ij}^k = \Gamma_{ij}^k - \bar{\Gamma}_{ij}^k$). The estimate \eqref{eq:deriv} controls the evolution of the derivatives of $G$ 
on $V_1$,
and implies an estimate on these derivatives of the same form. Namely, from Proposition 27 of \cite{KotschwarTimeAnalyticity} there are positive constants $\epsilon$, $C_1$, and $L_1$
such that
\[
\sup_{x\in V_1}|\nabla^{(k)}G|(x, t) \leq C_1L_1^k(k-2)!
\]
for all $k\geq 0$
and $t\in (0, T]$ with $|t-t_0| < \epsilon$. Used in conjunction with Proposition 25 of \cite{KotschwarTimeAnalyticity}, it implies that, for such $t$, 
\begin{equation}\label{eq:derg}
\sup_{x\in V_1}|\overline{\nabla}^{(k)}g|(x, t) \leq C_2 L_2^k(k-2)!
\end{equation} for all $k\geq 0$ and some constants $C_2$, $L_2$ independent of $k$. 

By assumption, the expression of $\bar{g}$ in the chart $(U, \varphi)$ is analytic. Thus, regarding $\partial$ as a connection on $U$,
we can find a neighborhood $V_2 \subset V_1$ of $x_0$ such 
that $H = \partial - \delb$  (i.e., $H_{ij}^k = -\overline{\Gamma}_{ij}^k$)
satisfies $|\partial^{(k)}H|\leq C_3L_3^k(k-2)!$ on $V_2$ for some $C_3$ and $L_3$ independent of  $k$. (We may use any of the metrics for the norm
 in this estimate, since they are all uniformly equivalent on $V_2$.) Taking this estimate together with \eqref{eq:derg} and applying Proposition 25 of \cite{KotschwarTimeAnalyticity}
 again with the connections $\partial$ and $\delb$,
 we obtain at last constants $C_4$ and $L_4$ such that
 \[
\sup_{x\in V_2}|\partial^{(k)}g|(x, t)\leq C_4 L_4^k(k-2)!
\]
for all $k\geq 0$ and all $t\in (0, T]$ with $|t- t_0|<\epsilon$.
\end{proof}

\section{The isometry group near spatial infinity}
We now resume our discussion of shrinking solitons and revert to our previous notational conventions. We will use $(\Cc^{\Sigma}_{r_0}, \bar{g}, \bar{f})$ to denote a shrinking soliton that is dynamically
asymptotic to $\Cc^{\Sigma}$,  and $g = g(\tau)$ to denote the associated solution to \eqref{eq:brf} on $\Cc^{\Sigma}_{r_0}\times [0, 1]$ with $g(0) = \hat{g}$ and $g(1) = \bar{g}$.
As before, $f = f(\tau)$ will denote the evolving family of potentials on $\Cc^{\Sigma}_{r_0}\times (0, 1]$. 

We will break the proof of Theorem \ref{thm:isomp} into several smaller pieces.
First we show that a local isometry $\phi:\Cc^{\Sigma}_{r_0}\to \Cc^{\Sigma}_{r_0}$ of any $g(\tau)$ 
will also preserve $\nabla f$, unless 
$g$ is flat (and hence conical itself).  We already know this for $\tau = 0$, since, by  \eqref{eq:fid0}, any isometry of $(\Cc^{\Sigma}_0, \hat{g})$
will preserve $\nabla f = \frac{r}{2\tau}\pd{}{r}$. 
\begin{lemma}\label{lem:gradf} Suppose that $\phi:\Cc^{\Sigma}_{r_0} \to \Cc^{\Sigma}_{r_0}$
is a local isometry of $g = g(\tau)$ for some $\tau\in (0, 1]$.
Then either $\phi^*(\nabla f) = \nabla f$ or $g$ is flat.
\end{lemma}
\begin{proof}
 Write $h = \phi^*f - f$. Then $g$ and  $\phi^*f$ together satisfy \eqref{eq:grstau} on $\Cc^{\Sigma}_{r_0}$, and 
 subtracting the two instances of equation \eqref{eq:grstau} yields that $\nabla\nabla h = 0$. 
 In particular, $|\nabla h| = a$ for some constant $a \geq 0$. Suppose that $a > 0$.  We claim that $\Rm\equiv0$.
 
 Since $\nabla\nabla h= 0$, we have that $\operatorname{Rm}(\nabla h, \cdot, \cdot, \cdot) = 0$. Differentiating this equation
 shows that $\nabla \Rm(\nabla h, \cdot, \cdot, \cdot, \cdot) = 0$ and hence also that $\nabla_{\nabla h}\Rm = 0$. In particular,
 $|\Rm|^2$ is constant on the integral curves of $\nabla h$. However, $|\Rm(r, \sigma)| \leq C(n, K_0)/r^2$ by \eqref{eq:curvdecay}, so if we can show that every sufficiently remote
 point of $\Cc^{\Sigma}_{r_0}$ sits on an integral curve of $\pm \nabla h$ that remains in $\Cc^{\Sigma}_{r_0}$ and tends to spatial infinity, we can conclude that $\Rm$ must vanish identically.
 
 For this, note that, by \eqref{eq:fid0} we can find $r_1 > 2r_0$ (depending on our fixed $\tau$ and $N_0$) such that $\inf_{\Cc^{\Sigma}_{r_1}} f > \sup_{\partial\Cc^{\Sigma}_{2r_0}} f$.
 Fix $p\in \Cc^{\Sigma}_{r_1}$ and let $b = \langle \nabla f, \nabla h\rangle(p)$ and $c= f(p)$. If $b\geq 0$, define $X = \nabla h$. Otherwise, define $X = -\nabla h$. 
 Let $\sigma$ denote the maximally defined integral curve of $X$ in $\Cc^{\Sigma}_{r_0}$ with $\sigma(0)= p$.  Using \eqref{eq:grstau},
 we have
 \begin{align*}
    \frac{d^2}{ds^2} f(\sigma(s)) &= \left(\nabla\nabla f(\nabla h, \nabla h) + \nabla\nabla h(\nabla f, \nabla h)\right)(\sigma(s))\\
    &= -\Rc(\nabla h, \nabla h)(\sigma (s)) + \frac{|\nabla h|^2(\sigma(s))}{2\tau } = \frac{a^2}{2\tau}
 \end{align*}
where $\sigma(s)$ is defined.
Thus, $f(\sigma(s)) = a^2s^2/(4\tau) + |b|s + c$ for all $s$ such that $\sigma(s)\in \Cc^{\Sigma}_{r_0}$. In particular, 
$f(\sigma(s)) \geq c > \sup_{\partial\Cc^{\Sigma}_{2r_0}} f$ for $s\geq 0$, so $\sigma$ is defined and remains in $\Cc^{\Sigma}_{2r_0}$ at least for $s\in [0, \infty)$. 
Since $f(\sigma(s))\to \infty$ as $s\to \infty$, we must have
$r(\sigma(s)) \to\infty$ as $s\to \infty$ as well, and consequently $|\Rm(p)|^2 = \lim_{s\to \infty}|\Rm|^2(\sigma(s)) = 0$.

But $p$ was arbitrary, so $\Rm\equiv 0$ on $\Cc^{\Sigma}_{r_1}$. From the analyticity of $(\Cc^{\Sigma}_{r_0}, g)$
it follows that $\Rm\equiv 0$ on $\Cc^{\Sigma}_{r_0}$, and $g$ is flat as claimed.
\end{proof}

The key ingredient in the proof of Theorem 2.2 is the following application of the backward uniqueness theorem in \cite{KotschwarWangConical}, which 
implies that an isometry of the end of the cone is also an isometry of the shrinker on some neighborhood of infinity. 
\begin{proposition}\label{prop:isomp0}
If $(\Cc^{\Sigma}_{r_0}, g, f)$ is dynamically asymptotic to $(\Cc^{\Sigma}_{r_0}, \hat{g})$, then, for any isometry $\phi\in \operatorname{Isom}(\Cc^{\Sigma}_{0}, \hat{g})$ there is $r_1\geq r_0$
such that $\phi|_{\Cc^{\Sigma}_{r_1}} \in \operatorname{Isom}(\Cc^{\Sigma}_{r_1}, g)$.
\end{proposition}
\begin{proof}
Let $\tilde{g} = \phi^*g$, $\tilde{f} = \phi^*f$. Then $(\Cc^{\Sigma}_{r_0}, \tilde{g}, \tilde{f})$ is also dynamically asymptotic to $(\Cc^{\Sigma}_{r_0}, \hat{g})$.
In fact, since $\phi$ preserves $r$, and $\Phi_{\tau}(r, \sigma) = (r/\sqrt{\tau}, \sigma)$, the corresponding solution $\tilde{\Phi}_{\tau} = 
\phi^{-1}\circ \Phi_{\tau} \circ \phi$ to \eqref{eq:phisys} associated to $\delt \tilde{f}$
actually coincides with $\Phi_{\tau}$, so $\tilde{g}(\tau) =\phi^*g(\tau)$ and $\tau\tilde{f}(\tau) = \tau\phi^*f(\tau)$ converge smoothly to $\hat{g}$ and $r^2/4$, 
and the estimates \eqref{eq:curvdecay}, \eqref{eq:fid0} will hold with the same constants $K_0$ and $N_0$.
By Theorem 2.2 of \cite{KotschwarWangConical}, there is an $a_0 \geq r_0$ and $T_0\in (0, 1]$ such that $g \equiv \tilde{g}$
on $\Cc^{\Sigma}_{a_0}\times [0, T_0]$.  In particular,
\[
\phi^*(g(1)) = \tilde{g}(1) = (\Phi_{T_0}\circ\tilde{\Phi}_{T_0}^{-1})^*(g(1)) = g(1)
\]
on $\Phi_{T_0}(\Cc^{\Sigma}_{a_0}) = \Cc^{\Sigma}_{a_0/\sqrt{T_0}}$.
 \end{proof}

Inspecting the proof of Theorem 2.2 in \cite{KotschwarWangConical} reveals that $a_0$
can be bounded above in terms of $K_0$, $N_0$, and other parameters external to the proof and independent of $\phi$. However, the existence of $T_0$ is established indirectly, leaving open the 
possibility that it (and therefore the value of $r_1$ in Proposition \ref{prop:isomp0}) above could depend on $\phi$. We believe the argument could be reworked to give an effective
proof of the existence of $T_0$. However, for the application we are seeking, it is simpler to use the result proven in the last section to argue that $\phi^*g = g$
on all of $\Cc^{\Sigma}_{r_0}$.

\begin{proposition}\label{prop:isomext}
 Suppose that, for some $\tau_0 \in [0, 1]$ and $a\geq r_0$,  $\phi:\Cc^{\Sigma}_{a}\to \Cc^{\Sigma}_{r_0}$ is a local isometry of $g(\tau_0)$.  If $\phi^*g(\tau_1) = g(\tau_1)$
 on $\Cc^{\Sigma}_{b}$ for some $\tau_1\in [0, 1]$ and $b \geq a$, 
 then $\phi^*g(\tau_1) = g(\tau_1)$ on $\Cc^{\Sigma}_a$, too.
\end{proposition}
\begin{proof}
By Theorem \ref{thm:analyticstructure}, the atlas of $g(\tau_0)$-normal coordinates on $\Cc^{\Sigma}_{r_0}$ induces a real-analytic structure $\Ac$ on $\Cc^{\Sigma}_{r_0}$ relative to which
both $g(\tau_0)$ and $g(\tau_1)$ are real-analytic. Since $\phi$ is a local isometry of $g(\tau_0)$ on $\Cc^{\Sigma}_{a}$, it is real-analytic relative to $\Ac$, too. But then $h =\phi^*g(\tau_1)-g(\tau_1)$ is also real-analytic on $\Cc^{\Sigma}_{a}$. Since $h\equiv 0$ on $\Cc^{\Sigma}_{b}$, we must have $h\equiv 0$ on $\Cc^{\Sigma}_{a}$ as well.
\end{proof}

For the proof of the reverse inclusion in Theorem \ref{thm:isomp}, we will need to know that an isometry of
$g(\tau)$ for some $\tau \in (0, 1]$ must also preserve $\bar{g} = g(1)$. Since $(\Cc^{\Sigma}_{r_0}, g(\tau))$  is incomplete, we cannot appeal directly to the backward uniqueness theorem in \cite{KotschwarRFBU}. However, the situation here is more elementary to begin with. Since  $g(\tau)$ is self-similar for $\tau \in (0, 1]$,  and,
by Lemma \ref{lem:gradf},
 we may assume that the isometry preserves $\nabla_{g(\tau)}f(\tau)$, the problem reduces to that of the uniqueness of the solutions
to the ODE \eqref{eq:phisys}. 
\begin{proposition}
\label{prop:origisom} Let $a \geq r_0$ and suppose $\phi^*g(T) = g(T)$ on $\Cc^{\Sigma}_{a}$ for some $0 < T\leq 1$
and some injective local diffeomorphism $\phi:\Cc^{\Sigma}_a\to \Cc^{\Sigma}_a$. Then $\phi^*\bar{g} = \bar{g}$ on $\Cc^{\Sigma}_a$.
\end{proposition}
\begin{proof} Write $g = g(T)$, $f = f(T)$, and $\nabla = \nabla_{g(T)}$. If $g$ is Ricci-flat on $\Cc^{\Sigma}_a$, 
then $\Phi_{T}^*\Rc(\bar{g}) \equiv 0$, so $\bar{g}$ is Ricci-flat on $\Phi_T(\Cc^{\Sigma}_{a}) = \Cc^{\Sigma}_{a/\sqrt{T}}$, and, hence,
on all of $\Cc^{\Sigma}_{r_0}$, by the real-analyticity of $(\Cc^{\Sigma}_{r_0}, \bar{g})$. 
But then the solution $g(\tau)$ to \eqref{eq:brf} is static, so $\bar{g} = g(1) = g(T) = g$, and the claim holds. 

Suppose then that $g$ is not Ricci-flat.
By Lemma \ref{lem:gradf}, $\phi^*X = X$ on $\Cc^{\Sigma}_a$, where
$X = T\nabla f = \Phi_T^*(\delb \bar{f}) = \delb\bar{f}$. Then, $f\circ\phi = f + c_0$, and \eqref{eq:fid0} implies that we can choose $b_0$ and $b_1$ depending on $a$, $N_0$, and $T$ such that $a \leq b_0 \leq b_1$
and
\[
\phi(\Cc^{\Sigma}_{b_1})\subset \Cc^{\Sigma}_{b_0}\subset \phi(\Cc^{\Sigma}_a).
\]
Since $\Phi_{\tau}(\Cc^{\Sigma}_{b_0}) = \Cc^{\Sigma}_{b_0/\sqrt{\tau}} \subset \Cc^{\Sigma}_{b_0}$ for all $\tau \in (0, 1]$, the map $\Psi_{\tau} = \phi^{-1}\circ \Phi_{\tau} \circ \phi$
is well-defined on $\Cc^{\Sigma}_{b_1}\times (0, 1]$. But $\phi^*X = X$ on $\Cc^{\Sigma}_{a}$, so $\Psi_{\tau}$ and $\Phi_{\tau}$ satisfy the same initial-value
problem and therefore coincide on $\Cc^{\Sigma}_{b_1}\times (0, 1]$. Thus $\phi \circ \Phi_{\tau} = \Phi_{\tau} \circ \phi$ on $\Cc^{\Sigma}_{b_1}\times (0, 1]$,
Then, on $\Cc^{\Sigma}_{b_1/\sqrt{T}} = \Phi_T(\Cc^{\Sigma}_{b_1})$, we have  $\phi^*\bar{g} = \bar{g}$ since
\begin{align*}
  \phi^*\bar{g} &=  T^{-1}\phi^*\left(\Phi^{-1}_{T}\right)^*g(T) =  T^{-1}\left(\Phi^{-1}_{T}\right)^*\phi^*g(T) =T^{-1}\left(\Phi^{-1}_{T}\right)^*g(T) = \bar{g}.
\end{align*}
But then Proposition \ref{prop:isomext} implies that $\phi^*\bar{g} = \bar{g}$ on $\Cc^{\Sigma}_a$ as well.
\end{proof}
Now we assemble the proof of Theorem 2.2 from the pieces above.
\begin{proof}[Proof of Theorem \ref{thm:isomp}]
Suppose first that $\phi \in \Isom(\Cc^{\Sigma}_{r_0}, \bar{g})$. We may assume that $\bar{g}$ is not flat, since, otherwise, the solution $g(\tau)$ to \eqref{eq:brf} associated to $\bar{g}$ is static and $\bar{g} = g(1) = g(0) = \hat{g}$, and the conclusion follows trivially.

Then we must have $\phi^*\delb \bar{f} = \delb \bar{f}$ by Lemma \ref{lem:gradf}. Since $\phi$ is a diffeomorphism of $\Cc^{\Sigma}_{r_0}$,
$\phi^{-1}\circ \Phi_{\tau}\circ\phi$ is well-defined on $\Cc^{\Sigma}_{r_0}\times (0, 1]$ and satisfies the same equation as $\Phi_{\tau}$.
This implies that $\phi$ commutes with $\Phi_{\tau}$ on $\Cc^{\Sigma}_{r_0}\times (0, 1]$, so
\[
 \phi^*g(\tau) = \tau\phi^*\Phi^{\ast}_{\tau}\bar{g} = \tau\Phi_{\tau}^*\phi^*\bar{g} = \tau \Phi_{\tau}^*\bar{g} = g(\tau)
\]
for all $\tau \in (0, 1]$. By continuity, we then conclude that
\[
 \phi^*\hat{g} = \lim_{\tau\to 0}\phi^*g(\tau) = \lim_{\tau\to 0} g(\tau) = \hat{g}
\]
on $\Cc^{\Sigma}_{r_0}$. As we have observed in the proof of Theorem \ref{thm:isom}, any diffeomorphism of $\Cc^{\Sigma}_{r_0}$ onto itself which preserves $\hat{g}$ must
preserve $r$, so $\phi$ must be of the form $\operatorname{Id}\times \gamma$ for some $\gamma\in \Isom(\Sigma, g_{\Sigma})$.

Now suppose that $\phi= \operatorname{Id}\times \gamma \in \Isom(\Cc^{\Sigma}_{r_0}, \hat{g})$. By Proposition \ref{prop:isomp0}, there is $r_1 \geq r_0$ such that $\phi^*\bar{g} = \bar{g}$
on $\Cc^{\Sigma}_{r_1}$. Proposition \ref{prop:isomext} then implies that $\phi^*\bar{g} = \bar{g}$ on $\Cc^{\Sigma}_{r_0}$.
\end{proof}

 \section{Extension of isometries in the complete case}

A classical theorem of Myers \cite{Myers} states that, if $(M, g)$, $(N, h)$ are complete, connected, simply connected real-analytic manifolds and
$\varphi: U\to V$ is an isometry between connected open sets $U\subset M$ and $V\subset N$, then $\varphi$ extends to an isometry $\Phi:M\to N$.
Here we show that the hypothesis of simple-connectivity can be exchanged for the assumption that the fundamental group of $U$ surjects onto that of $M$. 
As we have seen, this technical improvement has has some useful applications to the extension of isometries on complete asymptotically K\"ahler shrinkers
and its proof
requires just a few modifications to the original argument,
The ingredients
for the proof (if not the proof itself) are classical, see, e.g., \cite{Helgason}, \cite{KobayashiNomizu}, \cite{Nomizu}; as we are not aware of an appropriate
reference, we provide a proof here for completeness.

First, we recall some definitions from \cite{Helgason}. Two isometries $\varphi_i:U_i\to V_i$, $i=1$, $2$ are said to be \emph{immediate continuations} 
if $U_1\cap U_2\neq\emptyset$ and $\varphi_1 = \varphi_2$ on $U_1\cap U_2$. Given an isometry
$\varphi: U\to V$ and a continuous curve $\gamma:[0, 1]\to M$ with $\gamma(0) \in U$, a \emph{continuation} of $\varphi$ along $\gamma$
is defined to be a collection $\{\varphi_t\}_{t\in [0, 1]}$ of isometries $\varphi_t: U_t\to V_t$ between open subsets $U_t\subset M$, $V_t\subset N$ satisfying that $\varphi_0 = \varphi$
and that $\varphi_{t_1}$ and $\varphi_{t_2}$ are immediate continuations whenever $|t_1 - t_2|$ is sufficiently small.

The following statement is contained in Propositions 11.3-11.4 of \cite{Helgason}.
\begin{proposition}
\label{prop:pathcont}
Let $(M, g)$ and $(N, h)$ be complete, real-analytic Riemannian manifolds and $\varphi:U\to V$ an isometry between connected open sets $U\subset M$, $V\subset N$.
Then, 
\begin{enumerate}
\item[(a)] $\varphi$ admits a continuation $\{\varphi_t\}_{t\in [0, 1]}$ along any continuous $\gamma:[0, 1]\to M$ with $\gamma(0) \in U$. 
\item[(b)] If $\sigma:[0, 1]\to M$ is continuous and path-homotopic to $\gamma$ and $\{\psi_t\}_{t\in [0, 1]}$ is a continuation of $\varphi$ along $\sigma$, then $\varphi_1 = \psi_1$ on a neighborhood
 of $\gamma(1) = \sigma(1)$.
 \end{enumerate}
\end{proposition}

The following theorem gives a sufficient condition for a local isometry on $U$ to extend a local isometry on $M$.
If it exists, it can be recovered by continuation along paths from some $x_0 \in U$; the only issue to check is that this is well-defined. 
The gist of the proof is this: 
Provided $\pi_1(x_0, U)$ surjects onto $\pi_1(M, x_0)$, any two paths between points $x_0$ and $x\in M$ will be homotopic to paths which may initially differ inside $U$ 
but which meet at some point in $U$
and coincide thereafter.  The above proposition guarantees that the continuations along the latter (hence the former) paths agree. 
\begin{theorem}\label{thm:globalext}
Suppose $(M, g)$ and $(N, h)$ are complete connected real-analytic Riemannian manifolds and $\varphi: U \to V$ is an isometry between connected open sets $U\subset M$ and $V\subset N$.
If, for some $x_0\in U$, the homomorphism $\iota_{*}: \pi_1(U, x_0) \to \pi_1(M, x_0)$ induced by inclusion $\iota:U\hookrightarrow M$ is surjective, then $\varphi$ can be extended to a surjective local isometry $\Phi: M\to N$. 

If, in addition, $\pi_1(V, \varphi(x_0))$ surjects onto $\pi_1(N, \varphi(x_0))$ (e.g., if $M$ and $N$ are diffeomorphic)
then $\varphi$ extends to an isometry $\Phi:(M, g)\to(N, h)$.
\end{theorem}
\begin{proof} Let $x\in M$ and $\gamma^{0}$, $\gamma^1:I\to M$ be any continuous paths with $\gamma^{i}(0) = x_0$ and $\gamma^{i}(1) = x$. Here $I = [0, 1]$.
By Proposition \ref{prop:pathcont}, $\varphi$ admits continuations 
$\{\varphi^{0}_t\}_{t\in I}$ and $\{\varphi^{1}_t\}_{t\in I}$ along $\gamma^{0}$ and $\gamma^{1}$. We claim that $\varphi_1^{0} = \varphi_1^{1}$ in a neighborhood of $x$.

By assumption, $\gamma^{0}\cdot \bar{\gamma}^{1}$ is path-homotopic to a continuous loop $\alpha:I\to M$ based at $x_0$ with $\alpha(I)\subset U$. (Here the bar denotes the reverse 
parametrization.)
Fix some homotopy $H:I\times I\to M$ with $H(0, t) = \alpha(t)$ and $H(1, t) =(\gamma^{0}\cdot\bar{\gamma}^{1})(t)$ for all $t\in I$, and $H(s, 0) = H(s, 1) = x_0$ for all $s\in I$,
and consider the paths $\beta^0$, $\beta^1:I\to M$ defined by
\begin{align*}
 \beta^{0}(t) &= \left\{\begin{array}{lr}
		    \alpha(t) = H(0, t) & t\in [0, 1/2]\\
		    H(2t-1, 1/2) & t\in [1/2, 1],
                   \end{array}\right.\\                   
\beta^{1}(t) &= \left\{\begin{array}{lr}
		    \bar{\alpha}(t) = H(0, 1-t) & t\in [0, 1/2]\\
		    H( 2t-1, 1/2) & t\in [1/2, 1].
                   \end{array}\right. 
\end{align*}

Now let $\{\psi^{0}_t\}_{t\in I}$ and $\{\psi^{1}_t\}_{t\in I}$ be any continuations of $\varphi$ along $\beta^0$ and $\beta^1$. Note first that
$\psi^0_{1/2}$ and $\psi^1_{1/2}$ must coincide with $\varphi$ in a neighborhood of $\beta^0(1/2)= \beta^1(1/2) = \alpha(1/2)$. Indeed, since $\beta^{0}([0, 1/2])\subset U$
and $\beta^{1}([0, 1/2]) \subset U$, we can form continuations of $\varphi$ along these portions of $\beta^0$ and $\beta^1$ simply by restricting $\varphi$
to sufficiently small neighborhoods of the points of the paths. By uniqueness, these continuations must agree locally with the continuations
$\psi^{0}_t$ and $\psi^{1}_t$ for $t\leq 1/2$ as claimed.  But, since the paths $\beta^0$ and $\beta^1$ agree for $t \in [1/2, 1]$,
the collections $\{\psi^{0}_t|_{U_t^0\cap U_t^1}\}_{t\in [1/2, 1]}$ and $\{\psi^{1}_t|_{U^0_t\cap U^1_t}\}_{t\in [1/2, 1]}$ are both
continuations of the common isometry $\psi^{0}_{1/2} = \psi^1_{1/2} = \varphi$ along the path $\beta^0|_{[1/2, 1]}$.  
So $\psi^0_1$ and $\psi^1_1$ agree on a neighborhood of $x$.

On the other hand, $\beta^0$ and $\beta^1$ are path-homotopic to $\gamma^0$ and $\gamma^1$, respectively, by their construction.
So, by Proposition \ref{prop:pathcont}, $\varphi^0_1$ and $\varphi^1_1$ must agree with $\psi^0_1$ and $\psi_1^1$ and, hence, with each other, near $x$.

We now define $\Phi:M\to N$ at $x\in M$ by $\Phi(x) = \varphi_1(x)$, where $\{\varphi_t\}_{t\in I}$ is a continuation of $\varphi$ 
along any continuous path $\gamma:I\to M$ connecting
$x_0$ to $x$.
The discussion above shows that $\Phi$ is well-defined on $M$. Since it agrees with a local isometry in a neigborhood of each point, it is smooth
and satisfies $\Phi^*h = g$. Since $\Phi$ is a local isometry, it is an open map, and since $(M, g)$ and $(N, h)$ are complete, $\Phi$ is also closed.
Since $N$ is connected, $\Phi$ is therefore surjective. 

For the last claim, assuming that $\pi_1(V, \varphi(x_0))\to \pi_1(N, \varphi(x_0))$ is also surjective, we may apply the above argument to obtain an extensions 
of both $\varphi$ and $\varphi^{-1}$ to surjective local isometries $\Phi:(M, g)\to (N, h)$ and $\Psi: (N, h) \to (M, g)$.
Since $\Psi\circ\Phi$ and $\Phi\circ \Psi$ agree with the identity maps on  $U$ and $V$ and are themselves local isometries, they must agree
with the identity map on all of $M$ and $N$. So $\Phi$ is injective, and hence an isometry of $(M, g)$ onto $(N, h)$.
\end{proof}

Theorem \ref{thm:isomg} is then an easy consequence of the above theorem.

\begin{proof}[Proof of Theorem \ref{thm:isomg}] Suppose that the homomorphism $\pi_1(V, x_0) \to \pi_1(M, x_0)$ induced by inclusion $V\hookrightarrow M$ is surjective.
By Theorem \ref{thm:isom}, there is $r_1 > 0$ and an end $W\subset V$ diffeomorphic to $\Cc^{\Sigma}_{r_1}$ such that $\Isom(\Sigma, g_{\Sigma})\cong\Isom(W, g|_W)$.
Now, $V$ is diffeomorphic to $\Cc^{\Sigma}_{r_0}$, and the inclusion of $W \hookrightarrow V$ induces an isomorphism of fundamental groups, so $\pi_1(W, x_1)$
surjects onto $\pi_1(M, x_1)$ for some $x_1\in W$. Then \ref{thm:globalext} implies that any isometry
$\phi\in \Isom(W, g|_W)$ can be extended to an element $\Phi\in \Isom(M, g)$. This extension is unique, so the correspondence it defines is injective, and it is a homomorphism
since if $\Phi$ and $\Psi$ are the extensions of $\phi$ and $\psi$, respectively, then the extension of $\phi\circ \psi$ will agree with $\Phi \circ \Psi$
on $W$ and hence on all of $M$.
\end{proof}

\end{document}